\documentclass{amsart}
\usepackage[english ]{babel}
\usepackage{amsmath,latexsym,amssymb,verbatim}
\usepackage{pdfsync}

\usepackage[all]{xy}

\usepackage{hyperref}

\newtheorem{theorem}{Theorem}[section]
\newtheorem{lemma}[theorem]{Lemma}
\newtheorem{proposition}[theorem]{Proposition}
\newtheorem{corollary}[theorem]{Corollary}
\newtheorem{observation}[theorem]{Observation}

\newtheorem{note}[theorem]{Note}

\newtheorem{Formula of adjoint functors}[theorem]{Formula of adjoint functors}

\newtheorem*{theorem-non}{Theorem}

\newtheorem{examples}[theorem]{Examples}
\newtheorem{definition}[theorem]{Definition}
\newtheorem{notation}[theorem]{Notation}

\newtheorem{Adjunction formula}[theorem]{\indent\sc Adjunction formula}

\DeclareMathOperator{\limi}{{lim}}
\newcommand{\ilim}[1]{\,\underset{#1}{\underset{\to}{\limi}}\,}
\newcommand{\plim}[1]{\,\underset{#1}{\underset{\leftarrow}{\limi}}\,}

\DeclareMathOperator{\Hom}{{Hom}}

\DeclareMathOperator{\Coker}{{Coker}}
\DeclareMathOperator{\Ker}{{Ker}}
\DeclareMathOperator{\Ima}{{Im}}

\DeclareMathOperator{\R}{{\mathcal R}}
\DeclareMathOperator{\RR}{{\mathcal R}}
\newcommand{\dosflechasa}[3][]{\xymatrix@1{\ar@<1ex>[r]^-{#2}
\ar@<-1ex>[r]_-{#3} & }}
\newcommand{\dosflechas}{{\xymatrix@1  {\ar@<1ex>[r]
\ar@<-1ex>[r] & }}}
\newcommand{\dosflechasab}[3][]{\xymatrix@1  @C10pt {\ar@<1ex>[r]^-{#2}
\ar@<-1ex>[r]_-{#3} & }}
\newcommand{\dosflechasb}{{\xymatrix@1 @C10pt {\ar@<1ex>[r]
\ar@<-1ex>[r] & }}}

\NoCompileMatrices

\input arrow.tex

\begin{document}

\title{Left-exact Mittag-Leffler  functors of modules}

\author{Adri\'an Gordillo-Merino, Jos\'e Navarro and Pedro Sancho}

\address{Departamento de Matem\'{a}ticas, Universidad de Extremadura, Avda. de Elvas s/n, 06006 Badajoz. Spain}
\email{adgormer@unex.es, navarrogarmendia@unex.es, sancho@unex.es}
\thanks{All authors have been partially supported by Junta de Extremadura and FEDER funds.}



\begin{abstract} 
Let $R$ be an associative ring with unit. This paper deals with various aspects of the category of functors of $\mathcal R$-modules; that is, the category of additive and covariant functors from the category of R-modules to the category of abelian groups. We give several characterizations of  left-exact Mittag-Leffler functors of $\mathcal R$-modules. 
\end{abstract}

\maketitle

\section{Introduction}

Various types of modules are defined or determined via certain functors associated with them: flat modules,  projective modules, injective modules, etc. 
The aim of this paper is to study the functors of modules associated with flat Mittag-Leffler modules, much in the spirit as in the theory developed in \cite{Amelia} and \cite{Mitchell} (for small categories).


\medskip

In a more precise manner, if $R\,$ is an associative ring with unit, we will say that $\mathbb{M}\,$ is an $\mathcal{R}$-\textit{module} (right $\mathcal{R}$-module) if $\mathbb{M}\,$ is a covariant additive functor from the category of $R$-modules (respectively, right $R$-modules) to the category of abelian groups.

Any right $R$-module $M\,$ produces an $\mathcal{R}$-module. Namely, the \textit{quasi-coherent $\RR$-module} $\mathcal M$ associated with a right $R$-module $M$ is defined by
$$\mathcal M(N)=M\otimes_R N,$$
for any $R$-module $N$. It is significant to note that the category of (right) $R$-modules
is equivalent to the category of quasi-coherent $\RR$-modules (Cor. \ref{2.4}).

On the other hand, given an $\mathcal R$-module $\mathbb M$,  $\mathbb M^*$ is the  right $\mathcal R$-module defined as follows:
$$\mathbb M^*(N):=\Hom_{\RR}(\mathbb M,\mathcal N),$$
for any right $R$-module $N$. 

If we consider an $R$-module $M$, then we will call $\mathcal M^*$ the module scheme associated with $M$---by analogy with the geometric framework of functors defined on algebras (\cite{Amelia}). Module schemes are projective $\R$-modules and left-exact functors and a module scheme $\mathcal M^*$ is quasi-coherent if and only if $M$ is a finitely generated projective $R$-module (Prop. \ref{schpro} and \ref{3.11}, respectively).

A relevant fact is that quasi-coherent modules and module schemes are reflexive, that is, the canonical morphism of $\RR$-modules $\mathcal M\to \mathcal M^{**}$ is an isomorphism (Thm. \ref{reflex}; see also \cite{Amelia} for a slightly different version of this statement).

The category of $ \R $-modules is not a locally small category.  Then, it is interesting to consider the following full subcategory.

\begin{definition} \label{DModSch} Let $\langle\text{\sl ModSch}\rangle$ be the full subcategory of the category of $\R$-modules whose objects are those  $\R$-modules $\mathbb M$ for which there exists an exact sequence of $\R$-module morphisms
$$\oplus_{i\in I} \mathcal N_i^*\to \oplus_{j\in J} \mathcal N_j^*\to \mathbb M\to 0.$$

\end{definition}

$\langle\text{\sl ModSch}\rangle$ is
a bicomplete, locally small and  abelian category (Thm \ref{T6.11}). Besides, $\langle\text{\sl ModSch}\rangle$ is the smallest full subcategory of the category of $\R$-modules containing the $\RR$-module $\RR$ that is stable by limits, colimits and isomorphisms (that is, if an $\RR$-module is isomorphic to an object of the subcategory then it belongs to the subcategory).

From now on, we will assume that there exists an epimorphism $\oplus_{j\in J}\mathcal N_j^*\to \mathbb M$ for the $\R$-modules  $\mathbb M$ considered. If $J$ is countable, we will say that
$\mathbb M$ is countably generated.

 We prove that 
$\mathbb M$ is a left-exact $\R$-module if and only if $\mathbb M$ is a direct limit of module schemes (Thm. \ref{Lmodule}). This is is a version of Grothendieck's representability theorem (\cite[A. Prop. 3.1]{Grot3}) and its proof follows from standard categorical arguments. 
  Likewise we prove that   $\mathbb M$ is an exact functor if and only if $\mathbb M$ is a direct limit of module schemes $\mathcal F_i^*$, where $F_i$ are free modules (Thm. \ref{SLazard}).
In particular, if $\mathbb M=\mathcal M$ is quasi-coherent, 
Lazard's Theorem follows: any flat $R$-module is a direct limit of free $R$-modules.

We then focus on the question of determining those $\R$-modules $\mathbb M$ that can be expressed as a union of module schemes. We see that this question is closely related to $\mathbb M$ being a left-exact Mittag-Leffler module.


\begin{definition} \label{DML} An $\R$-module $\mathbb M$ is an ML module if the natural morphism $$\mathbb M(\prod_{i\in I}N_i)\to 
\prod_{i\in I} \mathbb M(N_i),$$
is injective, for any set of $R$-modules $\{N_i\}$.
\end{definition}

An $R$-module $M$ is a flat Mittag-Leffler module if and only if $\mathcal M$ is a left-exact ML $\R$-module (see \cite[II 2.1.5]{Raynaud})). Inspired by previous work of Grothendieck (\cite{Grot}), the class of Mittag-Leffler modules was first introduced by Gruson and Raynaud (\cite{RaynaudG}, \cite{Raynaud}) in their study of flat and projective modules. Soon after that, Mittag-Leffler modules were studied in relation with different functorial properties: flat strict Mittag-Leffler modules are ``universally torsionless modules'' (\cite{Garfinkel}),  ``trace modules'' (\cite{Ohm}) or  ``locally projective modules'' (in the sense of \cite{Zimmermann-Huisgen}).  More recently, there is a renewed interest in these modules, as they have been proposed as a generalized notion of vector bundle (\cite{Drinfeld}) and also have appeared to play a role in several different problems of Algebra (\cite{Herbera}, \cite{Saroch}).  
 
\begin{theorem-non}
Let $\mathbb M$ be an $\R$-module. 
$\mathbb M$ is a left-exact ML $\R$-module if and only if  it  is a direct limit of submodule schemes. If $\mathbb M$ is countably generated, $\mathbb M$ is a left-exact ML $\R$-module if and only if
it is projective. 
\end{theorem-non}

In particular,  if $M$ is a countably generated $R$-module, $M$ is a flat Mittag-Leffler  module  if and only if there exists a sequence of 
submodule schemes of $\mathcal M$, 
$$ \mathcal N_0^*\subseteq \mathcal N_1^*\subseteq\cdots\subseteq \mathcal N^*_i\subseteq\cdots\subseteq \mathcal M,$$
such that $\mathcal M=\cup_{i\in\mathbb N}\mathcal N_i^*$, and $M$ is a flat Mittag-Leffler  module if and only if it is  projective (\cite[2.2.2]{RaynaudG}).

Kaplansky's theorem about projective $R$-modules can be generalized to projective $\R$-modules:

\begin{theorem-non} Every  projective $\R$-module is a direct sum of projective countably generated $\R$-modules.
\end{theorem-non}

Finally, we study SML $\R$-modules.

\begin{definition} \label{DSM} An  $\R$-module $\mathbb M$ is an SML  $\R$-module if the natural morphism $$\mathbb M(N)\longrightarrow \Hom_R(\mathbb M^*(R),N),\,  m\mapsto \tilde m, \text{ where } \tilde m(w)=w_N(m)$$ is injective, for any $\,R$-module $\,N$. 
\end{definition}

$M$ is a flat strict Mittag-Leffler module if and only if $\mathcal M$
is an SML module (see \cite{Garfinkel}, Thm. 3.2). An $\R$-module $\mathbb M$ is an SML module iff there exists a monomorphism $$\mathbb M\hookrightarrow \prod_{I}\mathcal R\,\,\text{ 
(Prop \ref{caca}).}$$
In particular, $M$ is a flat strict Mittag-Leffler module if and only if there exists a monomorphism $\mathcal M \hookrightarrow \prod_{I}\mathcal R$.

\begin{theorem-non} An $\R$-module $\mathbb M$ is a left-exact SML $\R$-module if and only if any of the following, equivalent  conditions are met:

\begin{enumerate}

\item $\,\mathbb M\,$ is a direct limit of submodule schemes, $\,\mathcal N_i^*\subseteq\mathbb M\,$ and the dual morphism
$\,\mathbb M^*\to \mathcal N_i\,$ is an epimorphism, for any $\,i$.

\item If $\mathbb M$ is reflexive, every $\R$-module morphism $f\colon \mathbb M^*\to \mathcal N$ factors through the quasi-coherent module associated with $\Ima f_R$.
\end{enumerate}

\end{theorem-non}

Now, assume  $\mathbb M=\mathcal M$ is quassi-coherent.  We can specify more in (1): Let $\{M_i\}_{i\in I}$ be the set of all finitely generated submodules of $M$, and $M'_i:=\Ima[M^*\to M_i^*]$. 
$M$ is a flat strict Mittag-Leffler module if and only if $\mathcal M=\ilim{i} \mathcal {M'}_i^*$, that is to say, 
the natural morphism
$$N\otimes_R M\to \lim \limits_{\rightarrow} \Hom_{R}(M_i',N)$$
is an isomorphism, for any right $R$-module $N$ (Cor. \ref{C5.9}). Besides, if $R$ is a local ring  we prove that $M$ is a flat strict Mittag-Leffler module  if and only if  it is equal to the direct limit of its  finite free direct summands. (2) means that $M$ is a locally projective $R$-module (Prop \ref{PAzumaya}) and it
is equivalent to say that the cokernel of any morphism $f\colon \mathcal M^*\to \mathcal R$ is quasi-coherent,
that means that $M$ is a trace module (Prop. \ref{trace2}).

In Algebraic Geometry it is usual to consider the category of covariant additive functors from the category of $R$-algebras to the category of abelian groups, whereas in this work we consider the category of functors defined on the category of $R$-modules. There is an adjunction between these categories (arXiv:1811.11487) and many of the results presented in this paper are also true for functors defined on the category of $R$-algebras, although the direct proof is usually more difficult.
Finally, let us mention that a similar functorial study of (non-left-exact) ML modules has been carried out in \cite{Pedro3}. An effort has been made to make this paper as self-contained as possible.


\section{The category of $\mathcal{R}$-modules}

Let $\,R\,$ be an  associative ring with unit.

\begin{definition}
A (left)  $\,\mathcal R$-{\sl module} is a covariant, additive functor $\,\mathbb M\,$ from the category of (left) $\,R$-modules to the category of abelian groups.

A {\sl morphism of $\,\mathcal R$-modules} $f\colon \mathbb M\to \mathbb M'$
 is a morphism of functors  such that the  morphisms $\,f_{N}\colon \mathbb M({N})\to
 \mathbb M'({N})\,$ are morphisms of groups, for any  $\,R$-module $\,{N}$.
\end{definition}

Throughout this paper, many definitions or statements are given with one module structure (left or right);  we leave to the reader the task of producing the corresponding definitions or statements by interchanging the left and right structures.

\medskip
Observe that, if $\,\mathbb M\,$ is a left $\,\RR$-module, then $\,\mathbb M(R)\,$ is naturally a right $\,R$-module: for any $\,r\in R$, consider the morphism of $\,R$-modules
$\,\cdot r\colon R\to R$, $r'\mapsto r'\cdot r$, and define
$$m\cdot r:=\mathbb M(\cdot r)(m), \text{ for any } m\in \mathbb M(R)\, .$$
If $\,f\colon \mathbb M\to \mathbb M'\,$ is a morphism of $\,\RR$-modules, then
$\,f_R\colon \mathbb M(R)\to \mathbb M'(R)\,$ is a morphism of right $R$-modules.

\medskip
Let us write $\,\Hom_{\mathcal R}(\mathbb M,\mathbb M')\,$ to denote the  family of morphisms of $\,\mathcal R$-modules from $\,\mathbb M\,$ to $\,\mathbb M'$.

\begin{definition}
The {\sl dual of an $\,\mathcal{R}$-module} $\,\mathbb{M}\,$ is the right $\,\mathcal R$-module $\,\mathbb M^*\,$  defined, for any right $\,R$-module $\,N$, as follows:
$$\mathbb M^*(N):=\Hom_{\RR}(\mathbb M,\mathcal N) \, .$$
\end{definition}

The $\RR$-modules $\mathbb M$ considered in this paper verify that $\mathbb M^*(N)$ is a set. It can be proved that $\mathbb M^*(N)$ is a set for any $\RR$-module $\mathbb M$ (see \cite[7.4]{Pedro3}).

\medskip
Kernels, cokernels and images of morphisms of $\,\mathcal R$-modules will always be regarded in the category of $\,\mathcal R$-modules, and it holds:
$$(\Ker f)({N})=\Ker f_{N},\ (\Coker f)({N})=\Coker f_{N},\ (\Ima f)({N})=\Ima f_{N} \ . $$

Besides, for any  upward directed system $\,\{\mathbb{M}_i\}_{i\in I}\,$ and any  downward directed system $\,\{\mathbb{M}_j\}_{j\in J}$,
$$ (\ilim{} \mathbb M_i)({N})=\ilim{} (\mathbb M_i({N}))\, , \quad
(\plim{} \mathbb M_j)({N})=\plim{} (\mathbb M_j({N})) \, .$$

\subsection{Quasi-coherent modules}

\begin{definition} The {\sl quasi-coherent} (left) $\,\mathcal{R}$-module $\,{\mathcal M}\,$ associated with a right $\,R$-module $\,M\,$ is defined by $${\mathcal M}({N}) := M \otimes_R {N} \, . $$
\end{definition}

Any morphism of (right) $R$-modules $f_R\colon M\to M'$ induces the morphism of $\mathcal R$-modules 
$f\colon \mathcal M\to \mathcal M'$ defined by $f_N(m\otimes n):=f_R(m)\otimes n$, for any
$R$-module $N$, $m\in M$ and $n\in N$.

\begin{definition} Let $\,\mathbb M\,$ be an $\,\mathcal R$-module. The {\sl quasi-coherent module associated with}
the right $\,R$-module $\,\mathbb M(R)\,$ will be denoted by $\mathbb M_{qc}$ $$\mathbb M_{qc}({N}):=\mathbb M(R)\otimes_R N \, .$$
\end{definition}

If $\,N\,$ is an $\,R$-module, for any elements $\,m \in \mathbb{M}(R)$ and $\,n\in N$, let us define $\,m \cdot n \in M(N)\,$ in the following way: 
$$m\cdot n := M(\cdot n)(m)\, , $$ where $\,\mathbb M(\cdot n)\colon
\mathbb M(R)\to \mathbb M(N)$ is the morphism of groups induced by the morphism of $\,R$-modules $\,\cdot n \colon R\to  N$, $\,r\mapsto r\cdot n$. 

There exists a natural morphism $\,\mathbb{M}_{qc} \to \mathbb{M}\,$ defined by:
$$\mathbb M_{qc}({N})= \mathbb M(R)\otimes_R N \longrightarrow \mathbb{M}(N) \, , \quad m\otimes n\mapsto m\cdot n \ . $$
for any $R$-module $N$.

\begin{proposition} \label{tercer}
Let  $\,\mathbb M\,$ be an $\,{\mathcal R}$-module and let $\,M\,$ be  a right $\,R$-module. The assignment $\,f\mapsto f_R\,$ establishes a bijection:
$${\rm Hom}_{\mathcal R} ({\mathcal M}, \mathbb M) = {\rm Hom}_R (M, \mathbb M(R))\, .$$
\end{proposition}

\begin{proof} Given a morphim of $R$-modules $M\to  \mathbb M(R)$, consider the induced morphism $\mathcal M\to \mathbb M_{qc}$ and the composition of the  morphisms $\mathcal M\to \mathbb M_{qc}\to \mathbb M$.

Any morphism of $\,\mathcal R$-modules $\,f\colon \mathcal M\to \mathbb M\,$ is determined by $\,f_R$: if $\,N\,$ is an $\,R$-module, let us consider $\,n\in N\,$ and the morphism of $\,R$-modules $\,\cdot n \colon R\to N,$ $r\mapsto r\cdot n$. The  commutativity of the following diagrams imply $\,f_N(m\otimes n)=\mathbb M(\cdot n)(f_R(m))$, for any $\,m\in M$,
$$\xymatrix @C=8pt {M=\mathcal M(R) \ar[r]^-{f_R} \ar[d]^-{\mathcal M(\cdot n)} & \mathbb M(R) \ar[d]^-{\mathbb M(\cdot s)}\\ M\otimes_R N=\mathcal M(N) \ar[r]^-{f_N} & \mathbb M(N)}
\quad \xymatrix @C=8pt { m \ar@{|->}[r]^-{f_R} \ar@{|->}[d]^-{\mathcal M(\cdot n)} & \qquad f_R(m) \ar@{|->}[d]^-{\mathbb M(\cdot n)}\\ m\otimes n \ar@{|->}[r]^-{f_N} & f_N(m\otimes n)=\mathbb M(\cdot n)(f_R(m))
}$$

\end{proof}

\begin{corollary} \label{tercerb} If $\,\mathbb M\,$ is an $\,\mathcal R$-module, there exists a  a functorial equality, for any quasi-coherent $\,\mathcal R$-module $\,\mathcal N$,
$$\Hom_{\mathcal R}(\mathcal N,\mathbb M) = \Hom_{\mathcal R}(\mathcal N,\mathbb M_{qc}) \, .$$
\end{corollary}

\begin{proof} 
The last equality follows from a repeated use of Proposition \ref{tercer}: 
$$\Hom_{\mathcal R}(\mathcal N,\mathbb M)=\Hom_R(N,\mathbb M(R))=\Hom_{\mathcal R}(\mathcal N,\mathbb M_{qc}) \ . $$
\end{proof}


\begin{corollary}\label{2.4}   The functors $\,\mathcal M \rightsquigarrow \mathcal M(R) \,$ and $\,M \rightsquigarrow \mathcal M\,$ establish an equivalence of categories
$$\text{ Category of quasi-coherent $\mathcal R$-modules }\,  \equiv \text{Category of right $R$-modules } \, . $$

In particular,
$${\rm Hom}_{\mathcal R} ({\mathcal M},{\mathcal M'}) = {\rm Hom}_R (M,M')	\, .$$
\end{corollary}

If $\,f\colon \mathcal M\to \mathcal N\,$ is a morphism of $\,\RR$-modules, then $\,\Coker f\,$ is the quasi-coherent module associated with $\,\Coker f_R$.

\subsection{Module schemes}

\begin{definition} The {\sl $\mathcal R$-module scheme} associated with an $\,R$-module $\,M\,$ is the $\,\mathcal{R}$-module  $\,\mathcal M^*$. \end{definition}

Observe that the module scheme $\,\mathcal{M}^*\,$ is precisely the functor of points of the $\,R$-module $\,M$: for any $\,R$-module $\,N$, in virtue of Corollary \ref{2.4},
$$\mathcal{M}^* (N) = \Hom_{\mathcal{R}}(\mathcal{M} , \mathcal{N}) = \Hom_R (M , N) \, .  $$

\begin{proposition} \label{schpro} Module schemes   $\mathcal N^*$ are projective $\R$-modules and left-exact functors.
\end{proposition} 

\begin{proof} By Yoneda's Lemma, 
$\Hom_{\mathcal R}(\mathcal N^*,\mathbb M)=\mathbb M(N),$
for any $\R$-module $\mathbb M$.
Therefore, the functor  $\Hom_{\mathcal R}(\mathcal N^*,-)$ is exact. By Corollary \ref{2.4}, $\mathcal N^*=\Hom_R(N,-)$, that is a left-exact functor.

\end{proof}

\begin{proposition}\label{L5.11} Let $\,\{\mathbb M_i\}_{i\in I}\,$ be a directed system of $\,\mathcal R$-modules. Then, for any $\,R$-module $\,N$,
$$\Hom_{\mathcal R}(\mathcal N^*, \lim_{\rightarrow} \mathbb M_i)=\lim_{\rightarrow} \Hom_{\mathcal R}(\mathcal N^*, \mathbb M_i)\, .$$
\end{proposition}

\begin{proof} It is a consequence of Yoneda's Lemma,
$$\Hom_{\mathcal R}(\mathcal N^*,\lim_{\rightarrow}\mathbb M_i)=(\lim_{\rightarrow} \mathbb M_i)(N)= \lim_{\rightarrow} (\mathbb M_i(N))=
\lim_{\rightarrow} \Hom_{\mathcal R}(\mathcal N^*, \mathbb M_i) \, .$$
\end{proof}


\begin{proposition} \label{trivial} Let $\,\mathbb M\,$ be a $\,\RR$-module and $\,\mathbb N\,$ be a right $\,\mathcal R$-module. There exists an isomorphism of groups
$$\Hom_{\mathcal R}(\mathbb M,\mathbb N^*)=\Hom_{\mathcal R}(\mathbb N,\mathbb M^*)\, .$$
\end{proposition}

\begin{proof} Any morphism of $\,\mathcal{R}$-modules $\,f\colon \mathbb{M} \to \mathbb{N}^*\,$ defines a morphism $\,\tilde f \colon \mathbb{N}\to \mathbb{M}^*\,$ as follows: $\,(\tilde f_{S'}(n))_S(m):=(f_S(m))_{S'}(n)$, for any  $\,n\in\mathbb N(S')$,  $\,m\in\mathbb M(S)\,$ and $R$-modules $S',S$.
This assignment is an isomorphism because, for any $\,R$-module $\,S\,$ and any right module $\,S'$,
$$\Hom_{grp}(\mathbb M(S),\Hom_{grp}(\mathbb N(S'), S'\otimes_R S))=
\Hom_{grp}(\mathbb N(S'),\Hom_{grp}(\mathbb M(S), S'\otimes_R S)) \, .$$
\end{proof}

\begin{definition} 
If $\,\mathbb{M}\,$ is an $\,\mathcal R$-module, let  $\,\mathbb M_{sch}\,$ be  the module scheme defined as follows:
$$\mathbb M_{sch}:=({\mathbb M^*}_{qc})^*.$$
\end{definition}




\begin{proposition} \label{1211} If $\,\mathbb M\,$ is an $\,\mathcal{R}$-module there exists a natural morphism
$$\,\mathbb M\longrightarrow \mathbb M_{sch}\, ,$$
and a functorial equality, for any module scheme $\,\mathcal N^*$:
 $$\Hom_{\mathcal R}(\mathbb M,\mathcal N^*)=
\Hom_{\mathcal R}(\mathbb M_{sch},\mathcal N^*) \, .$$
\end{proposition}

\begin{proof} The morphism $\,\mathbb M\to \mathbb M_{sch}\,$ is defined, on any $\,R$-module $\,S$, as follows: an element $\,m \in \mathbb{M}(S)\,$ defines a morphism $\, \tilde m \colon \mathbb{M}^*(R) \to S$ via the formula $\,\tilde m (w) :=w_S(m)$, so that there exists a map
$$\mathbb M(S) \longrightarrow \Hom_R(\mathbb M^*(R),S)=\Hom_{\RR}({\mathbb M^*}_{qc},\mathcal S)=\mathbb M_{sch}(S) \ . $$ 

The last equality is a consequence of both Proposition \ref{trivial} and Corollary \ref{tercerb}:

$\Hom _{\mathcal R}(\mathbb M, \!\mathcal N^*) \! \overset{\text{\ref{trivial}}} =\!
\Hom_{\mathcal R}(\mathcal N,\!\mathbb M^*)\! \overset{\text{\ref{tercerb}}}=\!\Hom_{\mathcal R}(\mathcal N,\! (\mathbb M^*)_{qc}){\overset{\text{\ref{trivial}}}=
\Hom_{\mathcal R}(\mathbb M_{sch},\mathcal N^*)}.$

\end{proof}

\subsection{Reflexivity of quasi-coherent modules and module schemes}

\begin{theorem} \label{prop4}
Let $\,M\,$ be a right $\,R$-module and let $\,M'\,$ be an $\,R$-module. Then, the map $\, m\otimes m'\mapsto \tilde{m\otimes m'}\,$ establishes an isomorphism $${M} \otimes_{R} {M'}={\Hom}_{\mathcal R} ({\mathcal M^*}, {\mathcal M'})\, ,$$
where $\,(\tilde{m\otimes m'})_{N}(w):=w_R(m)\otimes m'$, for any $\,w\in \mathcal M^*(N)$.
\end{theorem}


\begin{proof} As $\,\mathcal{M}^*\,$ is a functor of points, the statement readily follows applying Yoneda's Lemma: 
$${\Hom}_{\mathcal R} ({\mathcal M^*}, {\mathcal M'})=\mathcal M'(M)=M\otimes_RM' \ .$$
\end{proof}


\begin{note} \label{2.12N}  On the other hand, it is not difficult to prove that the morphism $$f=\sum_{i=1}^n m_i\otimes m'_i\in {\Hom}_{\mathcal R} ({\mathcal M^*}, {\mathcal M'})={M} \otimes_{R} { M'},$$ coincides with the the composition of the morphisms of $\,\RR$-modules
$\,\mathcal M^*\overset g\to \mathcal L\overset h\to \mathcal M'$, where $\,L\,$ is the free module with basis $\,\{l_1,\ldots,l_n\}$, $\,h_R(l_i):=m'_i\,$ for any $\,i$, and $\,g:=\sum_{i} m_i\otimes l_i\in
{\Hom}_{\mathcal R} ({\mathcal M^*}, {\mathcal L})={M} \otimes_{R} { L}$. 

With these notations, observe that $\,h\,$ factors through  the quasi-coherent module associated with the finitely generated $\,R$-module $\,\Ima h_R\subseteq M'$, and, hence, so does $\,f$.
\end{note}


If $\,\mathbb{M}\,$ is an $\,\mathcal R$-module, there exists a natural morphism $${\mathbb M}\longrightarrow {\mathbb M}^{**} \ ,$$ that maps an element $\,m\in \mathbb M(N)\,$ to $\,\tilde m\in {\mathbb M}^{**}(N) =\Hom_{\mathcal R}(\mathbb M^*,\mathcal N)$, that is defined as $\,\tilde m_{N'}(w):=w_N(m)$, for any $\,w\in \mathbb M^*(N')=\Hom_{\RR}(\mathbb M,\mathcal N')$.


\begin{theorem} \label{reflex}
For any right $\,R$-module $\,M\,$, the natural morphism $${\mathcal M}\longrightarrow {\mathcal M^{**}}$$ is an isomorphism.
\end{theorem}

\begin{proof} It is a consequence of Theorem \ref{prop4}:
$$\mathcal M^{**}(N)=\Hom_{\RR}(\mathcal M^*,\mathcal N) \overset{\text{\ref{prop4}}}=M\otimes_R N=\mathcal M(N) \, .$$
\end{proof}

\begin{proposition} An $\mathcal R$-module $\mathbb M$ is a module scheme iff
it is reflexive, projective and $\Hom_{\mathcal R}(\mathbb M, \lim \limits_{\rightarrow} \mathbb M_i)=\lim \limits_{\rightarrow}\Hom_{\mathcal R}(\mathbb M, \mathbb M_i),$
for any directed system $\,\{\mathbb M_i\}_{i\in I}\,$.
\end{proposition}

\begin{proof} $\Rightarrow)$ It follows from Proposition\ref{schpro}, Proposition 
\ref{L5.11} and Theorem \ref{reflex}.

$\Leftarrow)$ The functor $\mathbb M^*$ is right-exact and commute with direct sums. Then, $\mathbb M^*$ is quasi-coherent by a theorem of Watts, \cite{Watts} Th. 1., and $\mathbb M=\mathbb M^{**}$ is a module scheme.

\end{proof}


\section{Left-exact $\mathcal R$-modules}

\begin{proposition} \label{3.11} An $R$-module $\,M\,$ is a finitely generated projective $R$-module if and only if the quasi-coherent module $\,\mathcal{M}\,$ is a module scheme.

\end{proposition}

\begin{proof} If $\mathcal M\simeq \mathcal N^*$, then $M$ is a finitely generated $R$-module by Note \ref{2.12N}. The functor, $\Hom_{\R}(\mathcal M,-)\simeq \Hom_{\R}(\mathcal N^*,-)$ is exact since $\mathcal N^*$ is a projective $\R$-module, by Proposition \ref{schpro}. 
Given an epimorphism of $R$-modules $S\to T$ then the associated morphism
$\mathcal S\to \mathcal T$ is an epimorphism and the map 
$$\xymatrix{ \Hom_{R}(M,S) \ar@{=}[r] & \Hom_{\R}(\mathcal M,\mathcal S) \ar[r] &
 \Hom_{\R}(\mathcal M,\mathcal T) \ar@{=}[r]  &  \Hom_{R}(M,T)}$$
 is surjective. Therefore, $M$ is a  projective $R$-module.


Conversely, there exist an $R$-module $M'$ and an isomorphism $\, M\oplus M'\simeq R^n$. 
Hence, there exists an isomorphism $\mathcal M\oplus \mathcal M'\simeq \mathcal R^n$.
The natural morphism
${\mathcal M}\to {\mathcal M}_{sch}$ is an isomorphism since the diagram
$$\xymatrix{{\mathcal M}_{sch}\oplus {\mathcal M'}_{sch} \ar@{=}[r] & (\mathcal M\oplus \mathcal M')_{sch} \ar[r]^-{\sim}  & (\mathcal R^n)_{sch} \ar@{=}[d]  \\  & \mathcal M\oplus \mathcal M'  \ar[u] \ar[ul] \ar[r]^-{\sim} & \mathcal R^n}$$
is commutative.
\end{proof}

\begin{lemma} \label{lemar} Let $\,\mathbb M$ be an $\,\R$-module.
If  $\,\mathbb M\,$ is left-exact, then $\,\Hom_{\mathcal R}(-,\mathbb M)\,$ is a left-exact functor  from the category of module schemes to the category of abelian groups.
If  $\,\mathbb M\,$ is right-exact, then $\,\Hom_{\mathcal R}(-,\mathbb M)\,$ is a right-exact functor  from the category of module schemes to the category of abelian groups.

\end{lemma}

\begin{proof} Observe that $\Hom_{\mathcal R}(\mathcal N^*,\mathbb M)=\mathbb M(N)$ and the sequence of $\RR$-module morphisms $\mathcal N_1^*\to \mathcal N_2^*\to \mathcal N_3^*$ is exact in the category of module schemes iff  the sequence of $R$-module morphisms 
$N_3\to N_2\to N_1$ is exact.
\end{proof}

We will say that $\mathbb M$ is a left-exact $\R$-module if it is an $\R$-module and a left-exact functor.

\begin{lemma} \label{lemas} Let $\,\mathbb M\,$ be a left-exact  $\,\R$-module and  let $\,f\colon \mathcal N^*\to\mathbb M\,$ be an $\,\RR$-module morphism. If $0\neq m\in \Ker f_S\subseteq \mathbb M(S)$, then there exists 
a submodule  $N'\underset\neq \subset N$, such that $f$ (uniquely) factors through the induced morphism $\pi\colon \mathcal N^*\to\mathcal N'^*$ and $\pi_S(m)=0$ .\end{lemma}

\begin{proof} We can consider $\,m\in \Ker f_S\subseteq \mathcal N^*(S)=\Hom_{\RR}(\mathcal S^*,\mathcal N^*)$ as a morphism $\tilde m\colon \mathcal S^*\to \mathcal N^*$, and it holds $\,f\circ \tilde m=0$ and $\tilde m_S(Id_S)=m$ (where $Id_S\in \mathcal S^*(S)$ is the identity morphism). By Lemma \ref{lemar}, $\,f\,$ uniquely factors through the module scheme associated with $\,N':=\Ker \tilde m^*_R\subseteq N$, which is different from $N$ since $\tilde m^*_R\neq 0$ (since $\tilde m\neq 0$). Besides, the composite morphism $\mathcal S^*\overset{\tilde m}\to \mathcal N^*\overset\pi\to \mathcal N'^*$ is zero, hence
$\pi_S(m)=\pi_S(\tilde m_S(Id_S))=0$.

\end{proof}

Unfortunately the category of $R$-modules  is not small. If it were small then any $\R$-module would be a quotient $\R$-module of a direct sum of module schemes. 

\begin{theorem} \label{Lmodule} Let $\mathbb M$ be an $\R$-module and assume that there exists an epimorphism $\pi\colon \oplus_{i\in I} \mathcal W_i^*\to \mathbb M$. Then, $\mathbb M$ is a left-exact $\R$-module iff $\,\mathbb M$ is a direct limit of module schemes.\end{theorem}

\begin{proof} $\Leftarrow)$ $\mathbb M$ is left-exact since it is a direct limit of left-exact functors.

$\Rightarrow)$ Let $J$ be the set of all finite subsets  of $I$.
For each, $j\in J$, put $W_j:=\oplus_{i\in j} W_i$ and 
let $\pi_j$ be the composition $\mathcal W_j^*=\oplus_{i\in j} \mathcal W_i^*\hookrightarrow \oplus_{i\in I} \mathcal W_i^*\overset\pi\to \mathbb M$. Let $K$ be the set of all the pairs $(j,\mathcal V^*)$, where $j\in J$ and  $\mathcal V^*$  is a module scheme quotient of $\mathcal W_j^*$, in the category of module schemes, such that $\pi_j$ (uniquely) factors through the natural morphism $\mathcal W_j^*\to \mathcal V^*$. Given $(j,\mathcal V^*), (j',\mathcal V'^*)\in K$, we say that $(j,\mathcal V^*)\leq (j',\mathcal V'^*)$ if $j\subseteq j'$ and $\Ker[\mathcal W_j^*\to \mathcal V^*]\subseteq \Ker[\mathcal W^*_{j'}\to \mathcal V'^*]$, then we have
the obvious commutative diagram 
$$\xymatrix{\mathcal W_j^* \ar@{^{(}->}[r] \ar[d] &  \mathcal W_{j'}^* \ar[r]^-{\pi_{j'}}  \ar[d] & \mathbb M\\ \mathcal V^* \ar[r] & \mathcal V'^* \ar[ru] &}$$
Given  $(j,\mathcal V^*), (j',\mathcal V'^*)\in K$, put $j''=j\cup j'$, $\mathcal V_1^*:=\Ker[\mathcal W_j^*\to \mathcal V^*]$ (that is, $V_1=W_j/V$) and   $\mathcal V_1'^*:=\Ker[\mathcal W_{j'}^*\to \mathcal V'^*]$ (that is, $V'_1=W_{j'}/V'$) and let $\mathcal V''^*$  be the cokernel in the category of modules schemes of the obvious morphism $\mathcal V_1^*\oplus \mathcal V_1'^*\to \mathcal W^*_{j''}$.
By Lemma \ref{lemar}, $(j,\mathcal V^*), (j',\mathcal V'^*)\leq (j'',\mathcal V''^*)$. Hence, $K$ is an upward directed set.
Let us prove that $\ilim{(j,V)\in K} \mathcal V^*\simeq \mathbb M$.  

The natural morphism $\ilim{(j,V)\in K} \mathcal V^*\to \mathbb M$ is an epimorphism: 
Given $m\in \mathbb M(S)$ there exist $m'\in \oplus_{i\in I} \mathcal W_i^*(S)$
such that $\pi(m')=m$. Obviously, $m'\in \mathcal W^*_j(S)$, for some $j\in J$. Hence, 
if $V=W_j$, $m\in \Ima[\mathcal V^*(S)\to\mathbb M(S)]$. 

The natural morphism $\ilim{(j,\mathcal V^*)\in K} \mathcal V^*\to \mathbb M$ is a monomorphism: Given $$\bar m\in  \Ker[\ilim{(j,V)\in K} \mathcal V^*(S)\to \mathbb M(S)],$$ there exist $(j,V)\in K$ and $m\in \Ker[\mathcal V^*(S)\to\mathbb M(S)]$,
such that the equivalence class of $m$ is $\bar m$.
 By Lemma \ref{lemas}, there exists a submodule $V'\subseteq V\subseteq W_j$ such that the morphism
$\mathcal V^*\to \mathbb M$ factors through $\mathcal V'^*$ and $m\in 
\Ker[\mathcal V^*(S)\to\mathcal V'^*(S)]$. Hence, $\bar m=0$.

\end{proof}

For a characterization of left-exact functors in abstract categories,  see \cite{Adamek}.

\begin{corollary} \label{C3.7} An $\,R$-module $\,M\,$ is flat if and only if the quasi-coherent module $\,\mathcal M\,$ is a direct limit of $\,\mathcal R$-module schemes.
\end{corollary}

\begin{observation} \label{OLazard} Given $f=\sum_{i=1}^r n_i\otimes m_i\in \Hom_{\R}(\mathcal N^*,\mathcal M^*)=N \otimes_R M$, put $N':=\langle n_1,\ldots,n_r
\rangle \subseteq N$. Then, $f$  is the composite morphism of  the natural morphism $\mathcal N^*\to \mathcal N'^*$ and $g=\sum_{i=1}^r n_i\otimes m_i\in \Hom_{\R}(\mathcal N'^*,\mathcal M^*)=N' \otimes_R M$.  

Then, in Theorem \ref{Lmodule}, if $\mathbb M=\mathcal M$ is quasi-coherent, we can suppose in the proof of this theorem that $V$, $V'$, etc. are finitely generated modules. Then, $\mathcal M$ is a direct limit of module schemes $\mathcal V_j^*$, where $V_j$ are finitely generated $R$-modules.

\end{observation}

%

\begin{theorem} \label{SLazard}
Let $\mathbb M$ be an $\R$-module and assume that there exists an epimorphism $\pi\colon \oplus_{i\in I} \mathcal W_i^*\to \mathbb M$. Then, $\mathbb M$ is an exact $\R$-module iff $\,\mathbb M=\ilim{j\in J} \mathcal L_j^*$, where $L_j$ are free $R$-modules.
\end{theorem}

\begin{proof} $\Rightarrow)$ By Theorem \ref{Lmodule}, $\,\mathbb M$ is the direct limit of a directed system of module schemes  $\{\mathcal N_j^*, f_{jk}\}_{j\leq k\in J}$. 
Denote $f_j$ the natural morphism $\mathcal N_j^*\to \mathbb M$.
For any $j\in J$, there exist a free module $L_j$ and an epimorphism
$\mathcal L_j\to N_j$. Let $g'_j\colon \mathcal N_j^*\hookrightarrow \mathcal L_j^*$ be the associated morphism. By Lemma \ref{lemar}, there exists a morphism $g_j\colon \mathcal L_j^*\to \mathbb M$ such that the diagram
$$\xymatrix{\mathcal  N_j^*\ar@{^{(}->}[r]^-{g'_j} \ar[rd]_-{f_j} & \mathcal L_j^* \ar[d]^-{g_j}\\
& \mathbb M}$$
is commutative. By Proposition \ref{L5.11}, there exist $j\leq \phi(j)\in J$ and a morphism
$g''_j\colon \mathcal L_j^*\to \mathcal N_{\phi(j)}^*$, such that $f_{\phi(j)}\circ g''_i=g_j$. Again by Proposition \ref{L5.11}, taking a greater $\phi(j)$, if it is necessary,
we can suppose that $f_{j\phi(j)}$ is equal to the composite morphism $\mathcal N_j^*\overset{g'_j}\to \mathcal L_j^*\overset{g''_j}\to  \mathcal N_{\phi(j)}^*$. Let $J'$
be the upward directed set defined by $J'=J$ and $j_1<'j_2$ if $\phi(j_1)<j_2$.
Consider the directed system $\{\mathcal L_j^*,  g_{jj'}=g'_{j'}\circ f_{\phi(j)j'}\circ g''_{j}\}$. Reader can easily check that $\,\mathbb M=\ilim{j\in J} \mathcal L_j^*$.

\end{proof}

\begin{observation} If $\mathbb M=\mathcal M$ is quassi-coherent, in the proof of Theorem  \ref{SLazard} we can suppose that $N_j$ are finitely generated $R$-modules, by Observation \ref{OLazard}.
Then, we can suppose that $L_j$ are finite free $R$-modules. Hence, 
any flat $R$-module $M$  is a direct limit of finite free $R$-modules (Lazard's theorem, \cite[A6.6]{eisenbud}).

\end{observation}

\section{Left-exact ML $\mathcal R$-modules}

Recall the definition of ML $\R$-module (Definition \ref{DML}). Module schemes $\mathcal M^*$ are obviously left-exact ML $\R$-modules.

\begin{proposition} \label{4.1} Every $\R$-submodule of an ML $\R$-module is an ML $\R$-module. Infinite direct products of ML $\R$-modules are ML $\R$-modules. A direct limit of ML $\R$-submodules of an $\R$-module is an ML $\R$-module.

\end{proposition}

\begin{proof} Let us only check the last sentence. Put $\mathbb M=\ilim{i\in I}\mathbb M_i$, where $\mathbb M_i\subseteq \mathbb M$ is an ML $\R$-module, for any $i\in I$. Let $\{N_j\}_{j\in J}$ be a set of $R$-modules. Then, the composite morphism $\mathbb M_i(\prod_j N_j)\to \prod_j \mathbb M_i(N_j)\to \prod_j \mathbb M(N_j)$ is injective and taking $\ilim{i}$ the morphism
$\mathbb M((\prod_j N_j))\to \prod_j \mathbb M(N_j)$ is injective. Hence, $\mathbb M$ is an ML $\R$-module.

\end{proof}

\begin{proposition} \label{a} Let $\mathbb M$ be a left-exact ML $\R$-module and $f\colon \mathcal N^*\to \mathbb M$ a morphism of $\R$-modules. Then, $f$ factors through a submodule scheme of $\mathbb M$. Moreover, there exists the smallest submodule scheme of $\mathbb M$ containing  $\Ima f$.
\end{proposition}

\begin{proof} 
Let $\{N_i\}_{i\in I}$ be the set of submodules $N_i\subseteq N$ such that 
$f$ factors through the module scheme $\mathcal N_i^*$ associated with $N_i$. Consider the obvious exact sequence of morphisms
$$0\to \cap_{i\in I} N_i\to N\to \prod_{i\in I} N/N_i\,.$$
 Put $N'_i:=N/N_i$, for any $i\in I$, $N':=\prod_{i\in I} N'_i$ 
and $N'':=\cap_{i\in I} N_i$. Then, we have the exact sequence
$$0\to \mathbb M(N'')\to \mathbb M(N)\to \mathbb M(N').$$
Observe that $\mathbb M(N')=\mathbb M(\prod_{i\in I} N'_i)\subseteq \prod_{i\in I} \mathbb M(N'_i)$, then we 
have the exact sequence of morphisms of groups
$$0\to \mathbb M(N'')\to \mathbb M(N)\to \prod_{i\in I}\mathbb M(N'_i).$$
Therefore, we have the exact sequence of morphisms
$$0\to \Hom_{\RR}({\mathcal N''}^*,\mathbb M)\to \Hom_{\RR}(\mathcal N^*,\mathbb M)\to \prod_{i\in I} \Hom_{\RR}(\mathcal N'^*_i,\mathbb M).$$
Hence, $f$ factors through a morphism $g\colon {\mathcal N''}^*\to \mathbb M$, since $f\in 
\Hom_{\RR}(\mathcal N_i^*,\mathbb M)=\Ker[\Hom_{\RR}(\mathcal N^*,\mathbb M)\to  \Hom_{\RR}(\mathcal N'^*_i,\mathbb M)]$, for every $i$.
Obviously, $g$ does not factor through the module scheme associated with a 
proper submodule of $N''$. By Lemma \ref{lemas}, the morphism $g\colon \mathcal N''^*\hookrightarrow \mathbb M$ is a monomorphism. 
$ \mathcal N''^*$ is the smallest submodule scheme of  $\mathbb M$ containing $\Ima f$: If $\Ima f\subseteq \mathcal W^*\subseteq \mathbb M$, then the morphisms $\mathcal N^*\to \mathcal W^*$ again factors through 
$ \mathcal N''^*$, since $\mathcal W^*$ is a left-exact ML $\R$-module.

\end{proof}


\begin{theorem} \label{ML} Let $\mathbb M$ be an $\R$-module and assume there exists an epimorphism $\pi\colon \oplus_{l\in L}\mathcal W_l^*\to \mathbb M$. The following statements are equivalent

\begin{enumerate}
\item $\mathbb M$ is a left-exact ML module.

\item Every morphism of $\mathcal R$-modules $\mathcal N^*\to \mathbb M$ factors through an $\R$-submodule scheme of $\mathbb M$, for any right $R$-module $N$.

\item $\mathbb M$ is equal to a direct limit of  $\R$-submodule schemes.
\end{enumerate}
\end{theorem}

\begin{proof} The implication $ (1) \Rightarrow (2)$  is precisely Proposition \ref{a}.

$ (2) \Rightarrow (3)$ Given a morphism $f\colon \mathcal N^*\to \mathbb M$, let $\mathcal W^*$ be a submodule scheme of $\mathbb M$ containing $\Ima f$. By Proposition \ref{a}, there exists the smallest submodule scheme
$\mathcal V^*$ of $\mathcal W^*$ containing $\Ima f$. If $\mathcal W'^*$ is another submodule scheme of $\mathbb M$ containing $\Ima f$, consider a submodule scheme $\mathcal W''^*$ of $\mathbb M$ containing $\mathcal W^*$ and $\mathcal W'^*$ (observe we have a natural morphism $\mathcal W^*\oplus \mathcal W'^*\to \mathbb M$). Hence, the smallest submodule scheme of $\mathcal W'^*$ containing $\Ima f$  is equal to 
the smallest submodule scheme of $\mathcal W''^*$ containing $\Ima f$, which is equal to $\mathcal V^*$. Therefore, $\mathcal V^*$ is the smallest submodule scheme of $\mathbb M$ containing $\Ima f$.

 Let $J$ be the set of all finite subsets of $L$, and
$\mathcal V^*_j$ the smallest submodule scheme of $\mathbb M$  containing $\pi(\oplus_{l\in j}\mathcal W_l^*)$, for any $j\in J$. Then, $\ilim{j\in J}\mathcal V_j^*=\mathbb M$.

$ (3) \Rightarrow (1)$ $\mathbb M$ is left-exact since it is a direct limit of left-exact functors. By Proposition \ref{4.1}, $\mathbb M$ is an ML $\R$-module.
\end{proof}

\begin{definition} An $\R$-module $\mathbb M$ is called countably generated if there exists an epimorphism $\oplus_{i\in\mathbb N} \mathcal N_i^*\to \mathbb M$.
\end{definition}

By Note \ref{2.12N}, $\mathcal M$ is countably generated iff $M$ is countably generated.

\begin{proposition} \label{CR7} An $\R$-module $\mathbb M$ is a left-exact ML $\R$-module of countable type if and only if there exists a chain of submodule schemes of $\mathbb M$
$$\mathcal N_1^*\subseteq \mathcal N_2^*\subseteq \cdots\subseteq \mathcal N_r^*\subseteq \cdots\subseteq \mathbb M$$
such that $\mathbb M=\cup_{r\in\mathbb N} \mathcal N_r^*$. 

\end{proposition}

\begin{proof} It is an immediate  consequence of Theorem \ref{ML}.

\end{proof}

\begin{corollary} Let  $\mathbb M$ be a countably generated $\R$-module.
$\mathbb M$ is a left-exact $\R$-module if and only if it is projective. 
\end{corollary}

\begin{proof} $\Leftarrow)$ Consider an epimorphism $\oplus_{i\in\mathbb N} \mathcal N_i^*\to \mathbb M$. $\mathbb M$ is a direct summand of $\oplus_{i\in\mathbb N} \mathcal N_i^*$. Hence, $\mathbb M$  is left-exact  and it is an ML $\R$-module
by Proposition \ref{4.1}.

$\Rightarrow)$ By Proposition \ref{CR7}, there exists a chain of submodule schemes of $\mathbb M$
$$\mathcal N_1^*\subseteq \mathcal N_2^*\subseteq \cdots\subseteq \mathcal N_r^*\subseteq \cdots\subseteq \mathbb M$$
such that $\mathbb M=\cup_{r\in\mathbb N} \mathcal N_r^*$. We have the exact sequence of morphisms
$$0\to \oplus_{i\in\mathbb N} \mathcal N_i^*\overset{f}\to \oplus_{i\in\mathbb N} \mathcal N_i^*\overset{g}\to \mathbb M\to 0, $$
where $f(n_0,n_1,n_2,\ldots):=(n_0,n_1-n_0,n_2-n_1,\ldots)$ and  $g((n_i))=\sum_i n_i$.
The morphism $r\colon \oplus_{i\in\mathbb N} \mathcal N_i^*\to \oplus_{i\in\mathbb N} \mathcal N_i^*$, $r(n_0,n_1,n_2,\ldots):=(n_0,n_0+n_1,n_0+n_1+n_2,\ldots)$ is a retract
of $f$. Therefore, $\mathbb M$ is a direct summand of $\oplus_{i\in\mathbb N} \mathcal N_i^*$ and it is projective.

\end{proof}

In particular,  if $M$ is a countably generated, it is a Mittag-Leffler module iff  it is projective.

\begin{theorem} Let $\mathbb M$ be an $\R$-module and assume there exists an epimorphism $\pi\colon \oplus_{i\in I}\mathcal N_i^*\to \mathbb M$. If $\mathbb M$
is projective then it is a direct sum of  (projective) countably generated $\R$-modules.
\end{theorem}

\begin{proof} Let us repeat the arguments given by Kaplansky in \cite{Kaplansky}.
$\mathbb M$ is a direct summand of $\mathbb N:=\oplus_{i\in I}\mathcal N_i^*$.
Put $\mathbb N=\mathbb M\oplus \mathbb M'$.
Let us inductively construct  a well ordered increasing sequence of $\R$-submodules $\{\mathbb N_\alpha\}$ of $\mathbb N$ such that

\begin{enumerate} 
\item If $\alpha$ is a limit ordinal, $\mathbb N_\alpha=\cup_{\beta<\alpha} \mathbb N_{\beta}$.

\item $\mathbb N_{\alpha+1}/\mathbb N_\alpha$ is countably generated.

\item Each $\mathbb N_\alpha$ is the direct sum of a subset of the $\mathcal N_i^*$'s.

\item $\mathbb N_\alpha=\mathbb M_\alpha\oplus \mathbb M'_\alpha$, where
$\mathbb M_\alpha = \mathbb N_\alpha\cap \mathbb M$ and $\mathbb M'_\alpha= \mathbb N_\alpha \cap \mathbb M'$.

\end{enumerate}
In this situation, $\mathbb M_\alpha$ is a direct summand of $\mathbb N$, since it is a direct summand of $\mathbb N_\alpha$, which is a direct summand of $\mathbb N$. Therefore, $\mathbb M_\alpha$ is a direct summand of $\mathbb M$. Besides, 
$$\mathbb N_{\alpha+1}/\mathbb N_{\alpha}=\mathbb M_{\alpha+1}/\mathbb M_{\alpha}\oplus \mathbb M'_{\alpha+1}/\mathbb M'_{\alpha}$$
Hence, $\mathbb M_{\alpha+1}/\mathbb M_{\alpha}$ is countably generated. Now it is clear that $\mathbb M=\cup_\alpha \mathbb M_\alpha$ and
$\mathbb M=\oplus_\alpha (\mathbb M_{\alpha+1}/\mathbb M_\alpha)$. 

Now, we proceed to contruct the $\mathbb N_\alpha$. Consider the obvious projections $\pi,\pi'\colon \mathbb N\to \mathbb M,\mathbb M'$.
Given $\mathbb N_{\alpha}$ let us construt $\mathbb N_{\alpha+1}$. Choose $\mathcal N_j^*\not\subseteq \mathbb N_\alpha$ and put $\mathcal N^*_{1}=\mathcal N_j^*$. By Proposition \ref{L5.11}, there exists $j_1,\ldots,j_r\in I$, such that
$\pi(\mathcal N^*_{1})\oplus \pi'(\mathcal N^*_1)\subseteq \mathcal N^*_{j_1}\oplus\cdots \oplus\mathcal N^*_{j_r}$. Put $\mathcal N^*_{2}:=\mathcal N^*_{j_1}$,$\ldots, \mathcal N^*_{r+1}:=\mathcal N^*_{j_r}$. Now we repeat on $\mathcal N^*_{2}$ the treatment just given to $\mathcal N^*_{1}$. The result will be a new finite set $\mathcal N^*_{r+2}, \ldots,  \mathcal N^*_{s}$. We proceed successively in this way. Finally,
$\mathbb N_{\alpha+1}$ is taken to be the $\R$-submodule generated by $\mathbb N_\alpha$ and all the $\mathcal N^*_{n}$'s. That $\mathbb N_{\alpha+1}$ has all the properties we desire is plain.

\end{proof}

\begin{observation} If $f\colon \mathcal N^*\to \mathcal M$ is a monomorphism, then
$N$ is a finitely generated module: The dual morphism $f^*\colon \mathcal M^*\to \mathcal N$ factors through the quasi-coherent module associated with a finitely generated submodule $N'\subseteq N$, by Note \ref{2.12N}. The morphism
$\mathcal N^*\to \mathcal N'^*$ is a monomorphism since  the composition
$\mathcal N^*\to \mathcal N'^*\to \mathcal M$ is a monomorphism. By Lemma \ref{epi}, the inclusion morphism $N'\subseteq N$ is an epimorphism, that is, $N=N'$.


\end{observation}

\section{Left-exact SML  $\mathcal R$-modules}

Recall the definition of SML $\R$-module (Definition \ref{DSM}). Module schemes $\mathcal M^*$ are obviously left-exact SML $\R$-modules.

\begin{proposition} \label{caca} Let $\,\mathbb M\,$ be an $\,\R$-module. The following statements are equivalent:

\begin{enumerate}
\item $\,\mathbb M\,$ is an SML $\R$-module.

\item The natural morphism $\,\mathbb M\to \mathbb M_{sch}\,$ is a monomorphism.

\item There exists a monomorphism $\,\mathbb M\hookrightarrow \mathcal V^*$.

\item There exists a monomorphism $\,\mathbb M\hookrightarrow \prod_{I}\mathcal R$.

\end{enumerate}

\end{proposition}

\begin{proof} $(1)\iff (2)$ It is an immediate consequence of the definition of $\mathbb M_{sch}$.

$(2) \iff (3)$ It is obvious.

$(3) \Rightarrow (4)$ Consider an epimorphism $\oplus_{I} \mathcal R\to \mathcal V$. Taking dual $\mathcal R$-modules, we have a monomorphism 
$\mathcal V^*\hookrightarrow \prod_{I}\mathcal R$.  The composition of the monomorphisms $\mathbb M\hookrightarrow \mathcal V^*\hookrightarrow \prod_{I}\mathcal R$ is a monomorphism.

$(4) \Rightarrow (2)$ By Proposition \ref{1211}, the monomorphism $\mathbb M\hookrightarrow \prod_{I}\mathcal R$ factors through a morphism $\mathbb M\to \mathbb M_{sch}$, that has to be a monomorphism.

\end{proof}

\begin{corollary}\cite[5.3]{Garfinkel}
If $\,M\,$ is a flat strict Mittag-Leffler module, then $\,M\,$ is a pure submodule of an $\,R$-module $\,\prod_I R$.
\end{corollary}

\begin{proof} There exists a monomorphism $\,\mathcal M\hookrightarrow \prod_{I}\mathcal R$.
For any right $R$-module $S$,  the morphism $S\otimes_RM\to S\otimes_R (\prod_I R)$ is injective, since the composition
$S\otimes_RM\to S\otimes_R (\prod_I R)\to \prod_IS$ is injective.

\end{proof}

\begin{corollary} \label{C5.3} Any $\R$-submodule of an SML $\R$-module is an SML $\R$-module.
\end{corollary}

\begin{proposition} \label{prdsml} If $\mathbb M_i$ is a left-exact SML $\R$-module, for any $i\in I$, then $\prod_{i\in I} \mathbb M_i$  is a left-exact SML $\R$-module.

\end{proposition} 

\begin{proof} Obviously,  $\prod_{i\in I} \mathbb M_i$  is a left-exact functor.
Put $\mathbb M_i\subseteq \mathcal V_i^*$. Then, $\prod_{i\in I} \mathbb M_i\subseteq \prod_i \mathcal V_i^*=\mathcal V^*$, where $V=\oplus  V_i$.
By Proposition \ref{caca}, $\prod_{i\in I} \mathbb M_i$  is an SML $\R$-module

\end{proof}

\begin{lemma} \label{epi} If an $\RR$-module morphism   $\mathcal N^*\to \mathcal M^*$ is a monomorphism, then the dual morphism $\mathcal M\to \mathcal N$ is an epimorphism. 

\end{lemma} 

 \begin{proof} The category of modules schemes is anti-equivalent to the category of quasi-coherent modules.  Hence, the morphism $\mathcal M\to \mathcal N$ is an epimorphism in the category of quasi-coherent modules, therefore it is an epimorphism. 
 \end{proof}
 
\begin{proposition} \label{Ppsml} Let $\mathbb M$ be an SML $\R$-module. If
$i\colon \mathcal N^*\to \mathbb M$ is a monomorphism, then $i^*\colon \mathbb M^*\to \mathcal N$ is an epimorphism.

\end{proposition} 

\begin{proof} By Proposition \ref{caca}, there exists a monomorphism  $\mathbb M\hookrightarrow \prod_I\mathcal R$. Consider the monomorphisms
$\mathcal N^*\overset{i}\hookrightarrow  \mathcal M \hookrightarrow  \prod_I\mathcal R$.
Dually, the composite morphism
$$\oplus_I R\to \mathbb M^*\overset{i^*}\to \mathcal N$$
is an epimorphism, by Lemma \ref{epi}. Hence, $i^*$ is an epimorphism.

\end{proof}

\begin{theorem} \label{blabla} An $\R$-module $\mathbb M\,$  is a left-exact SML  $\R$-module iff $\,\mathbb M\,$ is a direct limit of submodule schemes, $\,\mathcal N_i^*\subseteq\mathbb M\,$ and the dual morphism
$\,\mathbb M^*\to \mathcal N_i\,$ is an epimorphism, for any $\,i$.
In particular,  an $\R$-module $\mathbb M$ is a left-exact SML $\R$-module  iff it is a left-exact ML $\R$-module and for any submodule scheme $\mathcal N^*\subseteq \mathbb M$ the dual morphism
$\,\mathbb M^*\to \mathcal N\,$ is an epimorphism.
\end{theorem}

\begin{proof} $\Rightarrow)$ By Proposition \ref{caca}, there exists a monomorphism
$\mathbb M\hookrightarrow \mathcal V^*$. By Proposition \ref{4.1}, $\mathbb M$ is an ML $\R$-module since $\mathcal V^*$ is an ML $\R$-module.  By Proposition \ref{a},
any morphism $\mathcal N^*\to \mathbb M$ factors through a submodule scheme of $\mathbb M$. Recall $\mathbb M(N)=\Hom_{\RR}(\mathcal N^*,\mathbb M)$, then
given, $m\in \mathbb M(N)$, there exists a morphism $f\colon \mathcal N^*\to \mathbb M$  
such that $f_N(Id_N)=m$ (where $Id_N\in \mathcal N^*(N)$ is the identity morphism). Hence, there exists a submodule scheme $\mathcal N_i^*$ of $\mathbb M$
such that $m\in \mathcal N_i^*(N)$. The family of submodule schemes of $\mathcal V^*$ is a set because the family of quotient modules of $V$ is a set. Then, 
the family of submodule schemes of $\mathbb M$ is a set and $\mathbb M$ is equal to the direct limit of its submodule schemes.

$\Leftarrow)$ Put $\mathbb M=\lim \limits_{\rightarrow} \mathcal N_i^*$, where $\mathcal N_i^*\subseteq \mathbb M$ and the dual morphism $\mathbb M^*\to \mathcal N_i$ is an epimorphism, for any $i\in I$. $\mathbb M$ is a left-exact $\R$-module since it is a direct limit of left-exact $\R$-modules.
Observe that the morphism $\mathbb M^*(R)\to N_i$ is surjective, for any $i$, then the morphism $\Hom_R(N_i,N)\to \Hom_R(\mathbb M^*(R),N)$ is injective and  
the morphism $\lim\limits_{\rightarrow} \Hom_R(N_i,N)\to \Hom_R(\mathbb M^*(R),N)$ is injective.
Then, the  composition

$$\aligned \mathbb M(N) & =\lim \limits_{\rightarrow} \mathcal N_i^*(N)=
\lim_{\rightarrow}\Hom_{R}(N_i,  N)\hookrightarrow
\Hom_{R}(\mathbb M^*(R),  N)\endaligned $$
is injective. Hence, $\mathbb M$ is an SML $\R$-module. 
\end{proof}

\begin{theorem} \label{TTTB} Let $\,\mathbb M\,$ be a reflexive $\,\R$-module. Then, $\mathbb M$ is a left-exact SML module iff 
every morphism $\,f\colon \mathbb M^*\to \mathcal N\,$ factors through the 
quasi-coherent module associated with $\Ima f_R$, for any right $\,R$-module $\,N$.
\end{theorem}

\begin{proof} $\Rightarrow)$ The dual morphism $f^*\colon \mathcal N^*\to \mathbb M$ factors trough a submodule scheme $\mathcal N'^*\overset i\subseteq \mathbb M$, by Proposition \ref{a}.
Dually, we have the morphisms $\mathbb M^*\overset{i^*}\to \mathcal N'\to \mathcal N$ and 
$\mathbb M^*\to \mathcal N'$ is an epimorphism, by Proposition \ref{Ppsml}. 
Hence, $\Ima f_R=\Ima[N'\to N]$ and $\mathbb M^*$ factors through the quasi-coherent modules associated to $\Ima f_R$ since the morphism $\mathcal N'\to \mathcal N$ factors through it.

$\Leftarrow)$ A morphism $g\colon \mathcal N^*\to \mathbb M$ is zero iff
the dual morphism $g^*\colon \mathbb M^*\to \mathcal N$ is zero, and this last morphism is zero iff $\Ima g^*_R=0$, that is, $g^*_R=0$. Therefore, the morphism
$\mathbb M(N)=\Hom_{\R}(\mathcal N^*,\mathbb M)\to \Hom_{\R}(\mathbb M^*(R), N)$ is injective and $\mathbb M$ is an SML module. Let us check that $\mathbb M$ is left-exact.
If $N_1\subseteq N_2$, then
the morphism $\mathbb M(N_1)\to \mathbb M(N_2)$ is injective, since 
$ \Hom_{\R}(\mathbb M^*(R), N_1)\to  \Hom_{\R}(\mathbb M^*(R), N_2)$ is injective.
Consider an exact sequence of $R$-module morphisms $0\to N_1\overset i\to N_2\overset j\to N_3$ and let $\mathcal N_1\overset{\tilde i}\to \mathcal N_2\overset{\tilde j}\to \mathcal N_3$ be the associated morphisms. It remains to prove that $\Ker \mathbb M(j)=\Ima\mathbb M(i)$. Observe that $\mathbb M(N)=\mathbb M^{**}(N)=\Hom_{\RR}(\mathbb M^*,\mathcal N)$. By the hipothesis, a morphism $f\colon \mathbb M^*\to \mathcal N_2$ satisfies $\tilde j\circ f=0$ iff there exists a morphism $g\colon \mathbb M^*\to \mathcal N_1$ such that $f=\tilde i\circ g$. We are done.

\end{proof}

\begin{corollary} \label{C5.9} Let $\,M\,$ be an $\,R$-module and let $\,\{M_i\}_{i\in I}\,$ be the set of all  finitely generated $\,R$-submodules of $\,M$, and $\,M'_i:=\Ima[M^*\to M_i^*]$. Then,  $\,M\,$ is a flat strict Mittag-Leffler module iff   the natural morphism
$$N\otimes_R M\to \lim \limits_{\rightarrow} \Hom_{R}(M_i',N)$$
is an isomorphism, for any right $\,R$-module $\,N$.
\end{corollary} 

\begin{proof} $\Rightarrow)$ 
Let $\{M_i\}_{i\in I}$ be the set of  finitely generated submodules of $M$. 
Let $L_i$ be a finite free module, $L_i\to M_i$ an epimorphism and $M_i\to M$  the inclusion morphism, for any $i\in I$. Let  $\pi_i\colon \mathcal L_i\to \mathcal M_i$  and $f_i\colon \mathcal M_i\to\mathcal M$ be the induced morphisms. Taking dual $\RR$-modules, we have the morphisms
$$\mathcal M^*\overset{f_i^*}\to \mathcal M_i^*\overset{\pi_i^*}\hookrightarrow \mathcal L_i^*.$$
$\Ima (\pi_i^*\circ f_i^*)_R=\Ima (f_i^*)_R=M'_i$. By Theorem \ref{TTTB}, the composite morphism $\pi_i^*\circ f_i^*$
factors through the natural morphism $\mathcal M_i'\to \mathcal L_i^*$,
which factors through the natural morphism $\mathcal M_i'\to \mathcal M_i^*$. 
Then, we have the morphisms
$$\mathcal M^*\to \mathcal M'_i\to \mathcal M_i^*.$$ Taking dual $\RR$-modules, we have  the morphisms $\mathcal M_i\to {\mathcal M'_i}^*\to \mathcal M$ and
${\mathcal M'}_i^*\to \mathcal M$ is a monomorphism. Hence, $\mathcal M=\lim \limits_{\rightarrow}{\mathcal M'}_i^*$.
 Therefore, $$N\otimes_R M={\mathcal M}(N)=\lim_{\rightarrow} {{\mathcal M'_i}^*}(N)=\lim_{\rightarrow} \Hom_{R}(M'_i,N),$$
for any right $R$-module $N$.

$\Leftarrow)$ $\mathcal M=\ilim{i} {\mathcal M'_i}^*$ and the morphisms
$\mathcal M^*\to \mathcal M'_i$ are epimorphisms. By Theorem \ref{blabla}, $M$ is a flat strict Mittag-Leffler module.

\end{proof}

\subsection{Other characterizations of flat strict Mittag-Leffler modules}

\begin{lemma} \label{lemar2} Let $\mathbb M$ be left-exact and a reflexive $\RR$-module. The cokernel of an $\R$-module morphism $f\colon\mathbb M^*\to \mathcal N$ is quasi-coherent iff $f$ factors through the quasi-coherent module associated with $\Ima f_R$.

\end{lemma} 

\begin{proof}  
$N':=\Ima f_R$ is the kernel of the morphism $N\to \Coker f_R=:N''$. 
Let $\pi\colon \mathcal N\to \mathcal N''$ be the associated morphism.
Observe that $\mathbb M(N)=\mathbb M^{**}(N)=\Hom_{\RR}(\mathbb M^*, \mathcal N)$. Then,  $f$ factors through $\mathcal N'$  iff 
$\pi\circ f=0$, since $\mathbb M$ is left-exact. The natural morphism $\mathcal N''\to \Coker f$ is an epimorphism. Then, $\pi\circ f=0$ iff $\mathcal N''=\Coker f$. Therefore, $f$ factors through $\mathcal N'$  iff $\mathcal N''=\Coker f$.

\end{proof}

\begin{theorem} \label{TTT} Let $\,M\,$ be an $\,R$-module. Then $\,M\,$ is a flat strict Mittag-Leffler module iff
the cokernel of 
any morphism $\,f\colon \mathcal M^*\to \mathcal R\,$
is quasi-coherent.
\end{theorem}

\begin{proof}  $\Rightarrow)$ It is a consequence of Theorem \ref{TTTB} and 
Lemma \ref{lemar2}.

$\Leftarrow)$ Let $f\colon \mathcal M^*\to \mathcal N$ be a  morphism of $\RR$-modules such that $f_R=0$. We have to prove that $f= 0$. By Note \ref{2.12N}, $f$ factors through the quasi-coherent module associated with a finitely generated submodule of $N$. Hence, we can suppose that $N$ is finitely generated. We proceed by induction on the number of generators of $N$. Suppose $N=\langle n\rangle$. Let $\pi\colon \R\to \mathcal N$ be  an epimorphism. By Theorem \ref{prop4}, there exists a morphism $g\colon \mathcal M^*\to \mathcal R$ such that
$f=\pi\circ g$. Observe that $N':=\Ima g_R\subseteq \Ker \pi_R$, since $f_R=0$, then the composite morphism $\mathcal N'\overset i\to \mathcal R\overset\pi\to \mathcal N$ is zero.  By the hypothesis and Lemma \ref{lemar2}, $g$ factors through a morphism $h\colon \mathcal M^*\to \mathcal N'$, then $f=\pi\circ g=\pi\circ i\circ h=0\circ h=0$.

Suppose $N=\langle n_1,\ldots,n_r\rangle$. Put $N_1:=\langle n_1\rangle$ and 
$N_2:=N/N_1$ and let $i\colon \mathcal N_1\to\mathcal N$ and $\pi\colon \mathcal N\to \mathcal N_2$ be the induced morphisms. By Theorem \ref{prop4}, the  sequence of morphisms
$$\Hom_{\mathcal R}(\mathcal M^*,\mathcal N_1)\to \Hom_{\mathcal R}(\mathcal M^*,\mathcal N) \to \Hom_{\mathcal R}(\mathcal M^*,\mathcal N_2)\to 0$$
is exact. The morphism $f\circ \pi$ is zero since $(\pi\circ f)_R=\pi_R\circ f_R=\pi_R\circ 0=0$
and the Induction Hypothesis. Hence, there exists a morphism $g\colon \mathcal M^*\to\mathcal N_1$ such that $f=i\circ g$. The morphism $g_R$ is zero,  since $0=f_R=(i\circ g)_R=i_R\circ g_R$ and $i_R$ is injective. 
By the Induction Hypothesis, $g=0$ and $f=i\circ g=0$.

  \end{proof}

\begin{lemma} \label{tontuno} Let $\,f\colon N\to M\,$ be a morphism of $R$-modules. If the induced morphism $\,\mathcal M^*\to \mathcal N^*\,$ is an epimorphism, then $\,f\,$ has a retraction of $\,R$-modules. \end{lemma}

\begin{proof} The  epimorphism $\,\mathcal M^*\to \mathcal N^*\,$ has a section since $\mathcal N^*$ is a  projective $\R$-module by Proposition \ref{schpro}. Hence, 
$\,f\,$ has a retraction since the category of $R$-modules is anti-equivalent to the category of $\R$-module schemes.
\end{proof}

 \begin{proposition} \label{MitFre} Let $\,R\,$ be a local ring (that is, a ring where the non-units form a two-sided ideal). 
An $\,R$-module $\,M\,$ is a flat strict Mittag-Leffler module  if and only if  it equals the direct limit of its  finite free direct summands.
\end{proposition}

\begin{proof} Let $\,\mathfrak m\,$ be the set of non-units of $R$ and $K:=R/\mathfrak m$, and assume $\,M\,$ is a flat strict Mittag-Leffler module, so that we can write $\,\mathcal M=\lim \limits_{\rightarrow} \mathcal N_i^*$, where the morphisms $\,\mathcal M^*\to \mathcal N_i\,$ are epimorphisms. By Note \ref{2.12N},
$N_i$ is finitely generated, for any $i\in I$. There exist a finite free (right) $R$-module $L_i$ and an epimorphism $L_i\to N_i$ such that the induced morphism $L_i/L_i \cdot \mathfrak m\to N_i/N_i \cdot \mathfrak m$ is an isomorphism of $K$-modules. By Theorem \ref{prop4},
the epimorphism $\mathcal M^*\to \mathcal N_i$ factors through a morphism $\mathcal M^*\to \mathcal L_i$. The morphism $M^*\to L_i$ is an epimorphism by Nakayama's lemma, since the morphism $M^*/M^*\cdot \mathfrak m\to L_i/L_i\cdot\mathfrak m\simeq N_i/N_i\cdot \mathfrak m $ is an epimorphism. 
Hence, the morphism $\mathcal M^*\to \mathcal L_i$ is an epimorphism. By Lemma \ref{tontuno}, the morphism $L_i^* \to M$ has a retraction, that is, $L_i^*$ is a direct summand of $M$. Finally,  any finitely generated submodule of $M$ is included in some submodule $N_i^*$ and $N_i^*\subseteq L_i^*$. Hence, $M=\lim \limits_{\rightarrow} L_i^*$.

Conversely, let $\{L_i\}_{i\in I}$ be the set of finite free summands 
of $M$. Thus, $\,\mathcal M=\lim \limits_{\rightarrow} \mathcal L_i\,$ and the morphisms $\,\mathcal M^*\to \mathcal L_i^*\,$ are epimorphisms, so that $\,M\,$ is a flat strict Mittag-Leffler module.
\end{proof}

 \begin{proposition}[\cite{Gobel} Th. 3.29] \label{MitFre2} Let $\,R\,$ be a principal ideal domain. 
An $\,R$-module $\,M\,$ is a flat strict Mittag-Leffler module  if and only if  it equals the direct limit of its  finite free direct summands.
\end{proposition}

\begin{proof} Assume $\,M\,$ is a flat strict Mittag-Leffler module, so that we can write $\,\mathcal M=\lim \limits_{\rightarrow} \mathcal N_i^*$, where the morphisms $\,\mathcal M^*\to \mathcal N_i\,$ are epimorphisms. By Note \ref{2.12N},
$N_i$ is finitely generated, for any $i\in I$. There exist a finite free $R$-module $L_i$ and an epimorphism $L_i\to N_i$. By Theorem \ref{prop4},
the epimorphism $\mathcal M^*\to \mathcal N_i$ factors through a morphism $\pi\colon \mathcal M^*\to \mathcal L_i$. By Theorem \ref{TTT} and Lemma \ref{lemar2}, $\pi$ factors through the quasi-coherent module associated with $L'_i=\Ima \pi_R$. Observe that $L'_i$ is a finite free module since it is a submodule of a finite free module and $R$ is a principal
ideal domain. 
By Lemma \ref{tontuno}, the morphism ${L'_i}^* \to M$ has a retraction, that is, ${L'_i}^*$ is a direct summand of $M$. Finally,  any finitely generated submodule of $M$ is included in some submodule $N_i^*$ and $N_i^*\subseteq {L'_i}^*$. Hence, $M=\lim \limits_{\rightarrow} {L'_i}^*$.

Conversely, proceed as in the previous proof.
\end{proof}

\begin{corollary}[\cite{Kaplansky}] Let $R$ be  a local ring or a principal ideal domain, 
Then, any projective $R$-module of countable type  is free.
\end{corollary}

\begin{proof}  Write $\,M=\langle m_i\rangle_{i\in\mathbb N}$.
By Proposition \ref{MitFre} (or \ref{MitFre2}), there exists a chain of finite free direct summands of $M$
$$ L_0\subseteq L_1\subseteq\cdots\subseteq L_n\subseteq \cdots\subseteq M$$
such that $\langle m_1,\ldots, m_n\rangle \subseteq L_n$, for any $\,n\in\mathbb N$, and $\,M=\cup_{n\in \mathbb N} L_n$. 
As $\,L_i\,$ is a direct summand of $\,L_{i+1}$, let us write $L_{i+1}=L_i\oplus L'_{i+1}$, with $\,L'_{0}=L_0$. It is then easy to check that 
$\,M\simeq \oplus_{n\in \mathbb N} L'_n$.
\end{proof}

Let us recall that Kaplansky also proved that any projective module over a local ring or a principal ideal domain is free, using that any projective module is a direct sum of countably generated projective modules (\cite{Kaplansky}).

\subsection{Other definitions of flat strict Mittag-Leffler modules}

Finally, let us prove some well-known  characterizations of flat strict Mittag-Leffler modules.

\begin{definition} An $R$-module $M$ is said to be locally projective if for any epimorphism $\pi\colon N\to M$ and any morphism $f\colon R^n\to M$ (for any $n\in \mathbb N$) there exists a morphism $s\colon M\to N$ such that $f=\pi\circ s\circ f$.
\end{definition}

\begin{proposition}[\cite{Azumaya} Prop 6.] \label{PAzumaya} An $R$-module $M$ is locally projective iff it is a flat strict Mittag-Leffler module.
\end{proposition}

\begin{proof} $\Rightarrow)$ Consider an $\mathcal R$-module morphism 
$f^*\colon \mathcal M^*\to \mathcal R$ (or equivalently, an $\R$-module morphism 
$f\colon \R\to \mathcal M$). Let $L$ be a free $R$-module and $\pi\colon \mathcal L\to\mathcal M$ an epimorphism. There exists an $\R$-module morphism $s\colon \mathcal M\to \mathcal L$ such that $f=\pi\circ s\circ f$.
Then, $f^*=f^*\circ s^*\circ \pi^*$ and $\Ima f^*=\Ima (f^*\circ s^*)$
since
$$\Ima f^*\supseteq \Ima (f^*\circ s^*)\supseteq \Ima (f^*\circ s^*\circ \pi^*)=\Ima f^*.$$
Hence, $\Coker f^*=\Coker(f^*\circ s^*)$, which is quasi-coherent
by Theorem \ref{TTT}.  Again by Theorem \ref{TTT}, $M$ is a flat strict Mittag-Leffler module.

$\Leftarrow)$  Let $\pi\colon \mathcal N\to \mathcal M$ be an $\R$-module epimorphism and $f\colon \mathcal R^n\to\mathcal M$ an $\R$-module morphism. The morphism $f$ factors through an $\R$-submodule scheme
$i\colon \mathcal W^*\subset \mathcal M$ and a morphism $f'\colon \mathcal R^n\to \mathcal W^*$, since $M$ is a flat Mittag-Leffler module. There exists a morphism $s'\colon \mathcal W^*\to \mathcal N$
such that $\pi\circ s'=i$ since $\mathcal W^*$ is a projective $\mathcal R$-module, by Proposition \ref{schpro}. The map
$$\Hom_{\R}(\mathcal M,\mathcal N)=\mathcal M^*(N)\to \mathcal W(N)=N\otimes_R W=\Hom_{\mathcal R}(\mathcal W^*,\mathcal N)$$
is surjective, by Proposition \ref{Ppsml}. Hence, there exists
an $\R$-module morphism $s\colon \mathcal M\to \mathcal N$
such that $s\circ i=s'$. Therefore, 
$$\pi\circ s\circ f=\pi\circ s\circ i\circ f'=\pi\circ s'\circ f'=i\circ f'=f$$
and $M$ is locally projective.

\end{proof}

\begin{proposition}[\cite{Azumaya} Prop 7.] Let $M$ be a flat strict Mittag-Leffler module and $ N\subseteq M$ a pure submodule. Then, $N$ is a flat strict Mittag-Leffler module and it is locally split in $M$, that is, for any finitely generated submodule $N'\subseteq N$ there exists an $R$-module morphism $r\colon M\to N$ such that $r(n')=n'$ for any $n'\in N'$.

\end{proposition}

\begin{proof} The induced morphism $i\colon \mathcal N\to \mathcal M$ is a monomorphism. Then, by 
Corollary \ref{C5.3}, $N$ is a flat strict Mittag-Leffler module. Given the submodule $N'\subseteq N$, there exists a submodule scheme $i'\colon \mathcal W^*\subseteq \mathcal N$ such  that $N'\subseteq W^*$. 
The map
$$\Hom_{\R}(\mathcal M,\mathcal N)=\mathcal M^*(N)\to \mathcal W(N)=W\otimes_R N\overset{\text{\ref{prop4}}}=\Hom_{\R}(\mathcal W^*,\mathcal N)$$
is surjective, by Proposition \ref{Ppsml}. Then,
there exists an $\R$-module morphism $t\colon \mathcal M\to \mathcal N$ such that $t\circ i\circ i'= i'$.  It is easy to check that $t_R(n')=n'$ for any $n'\in N'$.

\end{proof}

\begin{definition}[\cite{Garfinkel}] A module $M$ is a trace module if every $m\in M$
holds $m\in M^*(m)\cdot M$, where $M^*(m):=\{w(m)\in R\colon w\in M^*\}$    
\end{definition}

\begin{proposition}[\cite{RaynaudG} II 2.3.4] \label{trace2} $M$ is a trace module iff it is a  flat strict Mittag-Leffler module.
\end{proposition}

\begin{proof} Consider the canonical isomorphism $M\overset{\text{\ref{prop4}}}=\Hom_{\mathcal R}(\mathcal M^*,\mathcal R)$, $m\mapsto \tilde m$ (where $\tilde m(w):=w(m)$). Obviouslly, $\Ima \tilde m_R=M^*(m)$.
Let $I\subseteq R$ be an ideal, $\tilde m$ factors through $\mathcal I$ iff $m\in I\cdot M$, as it is easy to see taking into account the following diagram
$$\xymatrix{\Hom_{\mathcal R}(\mathcal M^*,\mathcal I) \ar[r] \ar@{=}[d]^-{\text{\ref{prop4}}} &
\Hom_{\mathcal R}(\mathcal M^*,\mathcal R) \ar@{=}[d]^-{\text{\ref{prop4}}} \\ I\otimes_R M \ar[r] & M}$$
Then, $\tilde m$ factors the quasi-coherent module associated with $\Ima \tilde m_R$ if and only if 
$m\in M^*(m)\cdot M$. We are done, by Lemma \ref{lemar2} and Theorem \ref{TTT} .

\end{proof}

\section{Appendix: Abelian subcategory generated by module schemes}

Let $\mathbb M$ be an $\R$-module and $\mathbb P=\oplus_{i\in I} \mathcal N_i^*$, then 
$$\Hom_{\R}(\mathbb P,\mathbb M)=\prod_{i\in I} \Hom_{\R}(\mathcal N_i^*,\mathbb M)=\prod_{i\in I} \mathbb M(N_i).$$
Hence, $\mathbb P$ is a projective $\R$-module. Observe that $\mathbb P$ is a left-exact functor and 
$\mathbb P=\oplus_{i\in I} \mathcal N_i^*\subseteq \mathcal N^*$, where
$N=\oplus N_i$. Therefore, $\mathbb P$ is a left-exact SML $\R$-module.

\begin{notation} An infinite direct sum of modules schemes will be often denoted by $\mathbb P$
($\mathbb P'$, $\mathbb P_1$, etc.)
\end{notation}

Recall the definition of $\langle\text{\sl ModSch}\rangle $ (Definition \ref{DModSch}).
Given $\mathbb M,\mathbb M'\in \langle\text{\sl ModSch}\rangle $, consider an epimorphism $\mathbb P\to \mathbb M$, then 
$$\Hom_{\RR}(\mathbb M,\mathbb M')\subseteq \Hom_{\RR}(\mathbb P,\mathbb M')$$
and $\Hom_{\RR}(\mathbb M,\mathbb M')$ is a set, that is, 
$\langle\text{\sl ModSch}\rangle $ is a locally small category.

\begin{examples} Quasi-coherent modules and module schemes
belong to $\langle\text{\sl ModSch}\rangle $.

If $\mathbb M$ is a left-exact functor and there exists an epimorphism
$\mathbb P\to \mathbb M$, then $\mathbb M$ belongs to $\langle\text{\sl ModSch}\rangle $, by Theorem \ref{Lmodule}.

Left-exact SML $\R$-modules belong to  $\langle\text{\sl ModSch}\rangle $, by Theorem \ref{blabla}.

\end{examples}

\begin{proposition} \label{6.2} Let $f\colon \mathbb M_1\to \mathbb M_2$ be an $\R$-module morphism. If $\mathbb M_1,\mathbb M_2\in \langle\text{\sl ModSch}\rangle $, then 
$\Coker f\in \langle\text{\sl ModSch}\rangle $.

\end{proposition}

\begin{proof} Consider the exact sequences
$$\mathbb P_1\overset{\pi_1}\to \mathbb M_1\to 0, \qquad 
\mathbb P_2'\overset{i_2}\to \mathbb P_2\overset{\pi_2}\to \mathbb M_2\to 0.$$
There exists an $\R$-module morphism $g\colon \mathbb P_1\to \mathbb P_2$ such
that $\pi_2\circ g=f\circ \pi_1$, since $\mathbb P_1$ is a projective  $\R$-module. It is easy to check that
$$\Coker[\mathbb P_1\oplus \mathbb P_2'\overset{g\oplus i_2}\longrightarrow \mathbb P_2]=\Coker f.$$
Therefore, $\Coker f\in \langle\text{\sl ModSch}\rangle $.

\end{proof}

\begin{proposition} \label{6.3} Let $f\colon \mathbb P\to \mathbb P'$ be an $\R$-module morphism.
Then, $\Ker f$ is a left-exact SML module and $\Ima f\in$ $\langle\text{\sl ModSch}\rangle $.\end{proposition}

\begin{proof}  $\Ker f$ is  left-exact since $\mathbb P$ and $\mathbb P'$ are left-exact.
By Corollary \ref{C5.3}, $\Ker f$ is an SML $\R$-module, since $\mathbb P$  is an SML $\R$-module.
By Proposition \ref{6.2}, $\Ima f=\Coker[\Ker f\to \mathbb P]$ belongs to 
$\langle\text{\sl ModSch}\rangle $.

\end{proof}

\begin{proposition} \label{6.3B} Let $0\to \mathbb M_1\to \mathbb M_2\overset\pi \to \mathbb M_3\to 0$ be an exact sequence of $\R$-module morphisms. If $\mathbb M_1,\mathbb M_3\in \langle\text{\sl ModSch}\rangle $ then $\mathbb M_2\in \langle\text{\sl ModSch}\rangle$.
\end{proposition}

\begin{proof} Consider exact sequences $ \mathbb P_1\overset{f_1}\to \mathbb P'_1\overset{f'_1}\to \mathbb M_1\to 0$ and $ \mathbb P_3\overset{f_3}\to \mathbb P'_3\overset{f_3'}\to \mathbb M_3\to 0$. There exits a morphism $g\colon \mathbb P'_3\to \mathbb M_2$ such that $\pi\circ g=f_3'$, since $\mathbb P'_3$ is projective.
Consider the exact sequence
$$\xymatrix{0\ar[r] & \mathbb M_1\ar[r] & \mathbb M_2\ar[r] ^-\pi & \mathbb M_3\ar[r] & 0\\ 0\ar[r] & \mathbb P'_1\ar[r] \ar[u]^-{f_1'} & \mathbb P'_1\oplus\mathbb P'_3 \ar[r]  \ar[u]^-{f_1'\oplus g}& \mathbb P'_3\ar[r]  \ar[u]^-{f_3'} & 0}$$
By the snake lemma, $f_1'\oplus g$ is an epimorphism and 
we have the exact sequence of morphisms
$$0\to \Ker f_1'\to \Ker(f_1'\oplus g)\to \Ker f'_3\to 0,$$ where $\Ker f'_1=\Ima f_1$ and $\ker f'_3=\Ima f_3$ belong to
$\langle\text{\sl ModSch}\rangle $ by Proposition \ref{6.3}. Then, again there exists
and epimorphism $\mathbb P''\to \Ker(f_1'\oplus g)$ and $\mathbb M_2\in \langle\text{\sl ModSch}\rangle$.

\end{proof}

\begin{proposition} \label{6.4B} Let $f\colon \mathbb P\to \mathbb M$ be an $\R$-module epimorphism.  If $\mathbb M \in \langle\text{\sl ModSch}\rangle $, then 
$\Ker f\in \langle\text{\sl ModSch}\rangle $.
\end{proposition}

\begin{proof} Consider an exact sequence of $\R$-module morphisms
$$
\mathbb P_1'\overset{i_1}\to \mathbb P_1\overset{\pi_1}\to \mathbb M\to 0.$$
There exists a morphism $g\colon \mathbb P\to \mathbb P_1$ such that 
$\pi_1\circ g=f$, since $\mathbb P$ is a projective $\R$-module. Consider the exact sequence of $\R$-module morphisms
$$\xymatrix{0 \ar[r] & \Ker f \ar[r] & \mathbb P \ar[r]^-f & \mathbb M \ar[r] & 0\\
0 \ar[r] & 0 \ar[r] \ar[u] & \mathbb P_1 \ar@{=}[r] \ar[u]^-g & \mathbb P_1  \ar[u]^-{\pi_1} \ar[r]  & 0}
$$
By the snake lemma, we have the exact sequence of $\R$-module morphisms
$$0\to \Ker g \to \Ker \pi_1 \to \Ker f\to \Coker g\to 0$$
Observe that $\Ker g$ and $\Ker \pi_1=\Ima i_1 $ belong  to $\langle\text{\sl ModSch}\rangle$, by Proposition \ref{6.3}. Therefore, $\Coker[\Ker g \to \Ker \pi_1 ]$  and $\Coker g$ belong to $\langle\text{\sl ModSch}\rangle$ by Proposition \ref{6.2}.
Hence, $\Ker f\in \langle\text{\sl ModSch}\rangle$ by Proposition \ref{6.3B}.

\end{proof}

\begin{proposition} \label{6.5} Let $f\colon \mathbb M\to \mathbb M'$ be an $\R$-module epimorphism.  If $\mathbb M,\mathbb M'\in \langle\text{\sl ModSch}\rangle $, then 
$\Ker f\in \langle\text{\sl ModSch}\rangle $.
\end{proposition}

\begin{proof} Consider an epimorphism $\pi\colon \mathbb P\to \mathbb M$
and the commutative diagram
$$\xymatrix{0 \ar[r] & \mathbb \Ker \pi \ar[r] \ar[d] & \mathbb P \ar@{=}[d]  \ar[r]^-\pi  & \mathbb M \ar[r] \ar[d]^-f & 0\\ 0 \ar[r] & \mathbb \Ker (f\circ \pi) \ar[r] & \mathbb P \ar[r]^-{f\circ \pi}  & \mathbb M' \ar[r] & 0}$$
By Proposition \ref{6.4B}, $\Ker \pi,\Ker(f\circ \pi)\in  \langle\text{\sl ModSch}\rangle $.
By the snake lemma, we have the exact sequence
$$0\to \Ker\pi\to \Ker (f\circ \pi)\to \Ker f\to 0$$
By the Proposition \ref{6.2}, $ \Ker f\in  \langle\text{\sl ModSch}\rangle $.

\end{proof}

\begin{proposition} \label{6.6} Let $\mathbb M,\mathbb M'\in \langle\text{\sl ModSch}\rangle $ and let 
$f\colon \mathbb M\to \mathbb M'$ be an $\R$-module morphism.  Then, 
$\Ker f\in \langle\text{\sl ModSch}\rangle $.
\end{proposition}

\begin{proof} By Proposition \ref{6.2}, $\Coker f\in \langle\text{\sl ModSch}\rangle $.
By Proposition \ref{6.5}, $\Ima f=\Ker[\mathbb M'\to \Coker f]\in \langle\text{\sl ModSch}\rangle $. By Proposition \ref{6.5}, $\Ker f=\Ker[\mathbb M\to \Ima f]\in \langle\text{\sl ModSch}\rangle $.
\end{proof}

\begin{proposition} If $\mathbb M_i\in \langle\text{\sl ModSch}\rangle $ for any $i\in I$, then $\oplus_{i\in I} \mathbb M_i\in \langle\text{\sl ModSch}\rangle $.
\end{proposition}

\begin{proof} It is obvious.
\end{proof}

\begin{proposition} \label{6.10} If $\mathbb M_i\in \langle\text{\sl ModSch}\rangle $ for any $i\in I$, then $\prod_i \mathbb M_i \in \langle\text{\sl ModSch}\rangle $.
\end{proposition}

\begin{proof} It is sufficient to prove that $\prod_{i\in I}\mathbb  P_i \in \langle\text{\sl ModSch}\rangle $. $\prod_{i\in I}\mathbb  P_i $ is a left-exact SML $\R$-module 
by Proposition \ref{prdsml}. We are done.
\end{proof}

Given $\mathcal M^*$, consider a free presentation of $M$, 
$\oplus_I R\to\oplus_J R\to M\to 0$. Then, we have an exact sequence $0\to \mathcal M^*\to \prod_J \mathcal R\to\prod_I\mathcal R$. Now, the following theorem is immediate. 

\begin{theorem} \label{T6.11} $\langle\text{\sl ModSch}\rangle $ is a bicomplete, locally small and  abelian category. Besides, $\langle\text{\sl ModSch}\rangle$ is the smallest full subcategory of the category of $\R$-modules containing $\mathcal R$  that is stable by kernels, cokernels, direct limits,  inverse limits and isomorphims (that is, if an $\RR$-module is isomorphic to an object of the subcategory then it belongs to the subcategory).
\end{theorem}

\begin{notation} Let $\langle\text{\sl RModSch}\rangle$ be the full subcategory of the category of right $\R$-modules whose objects are those  right $\R$-modules $\mathbb M$ for which there exists an exact sequence of $\R$-module morphisms
$$\mathbb P'\to \mathbb P\to \mathbb M\to 0$$
where $\mathbb P=\oplus_{i\in I} \mathcal N_i^*$ and $\mathbb P'=\oplus_{j\in J} \mathcal N_j^*$ (and $N_i,N_j$ are right $R$-modules).\end{notation}

\begin{proposition} If $\mathbb M\in \langle\text{\sl ModSch}\rangle $, then $\mathbb M^* \in  \langle\text{\sl RModSch}\rangle$.
\end{proposition}

\begin{proof}  Consider an exact sequence $\mathbb P'\to\mathbb P\to \mathbb M\to 0$. Dually, $0\to \mathbb M^*\to\mathbb P^*\to\mathbb P'^*$ is exact.
It is enough to prove that $\mathbb P^*, \mathbb P'^*\in  \langle\text{\sl RModSch}\rangle$.
Put $\mathbb P=\oplus_{i\in I}\mathcal N_i^*$. By Proposition \ref{6.10}, 
$\mathbb P^*=\prod_{i\in I}\mathcal N_i \in  \langle\text{\sl RModSch}\rangle$.
\end{proof}

%
%
%
%
%
%
%

\end{document}